\documentclass{amsart}
\usepackage{amsmath, amsthm, amscd, amssymb,latexsym,enumerate}
\usepackage{url}
\usepackage{array}
\usepackage[all]{xy}
\usepackage{enumitem}
\usepackage{hyperref}

\def\Q{{\mathbb{Q}}}
\def\R{{\mathbb{R}}}
\def\Z{{\mathbb{Z}}}
\def\C{{\mathbb{C}}}

\def\Frob{\mathrm{Frob}}
\def\Gal{\mathrm{Gal}}
\def\End{\mathrm{End}}

\def\Hom{\mathrm{Hom}}
\def\deg{\mathrm {deg}}
\def\dim{\mathrm {dim}}

\def\ker{\mathrm {ker}}
\def\fchar{\mathrm{char}}

\def\F{{\mathbb{F}}}
\def\M{{\mathcal{M}}}
\def\N{{\mathcal{N}}}
\def\O{{\mathcal{O}}}

\def\X{{{A}}}
\def\Y{{{B}}}
\def\p{{\mathfrak p}}

\def\bmu{{\boldsymbol \mu}}
\def\bbz{{\mathbb{Z}}}
\def\bbq{{\mathbb{Q}}}

\def\isom{\xrightarrow{\sim}}

\newtheorem{thm}{Theorem}[section]
\newtheorem{lem}[thm]{Lemma}
\newtheorem{cor}[thm]{Corollary}
\newtheorem{prop}[thm]{Proposition}
\theoremstyle{definition}
\newtheorem{defn}[thm]{Definition}
\newtheorem{ex}[thm]{Example}
\newtheorem{rem}[thm]{Remark}

\title[Isogenies of abelian varieties over finite fields]
{Isogenies of abelian varieties over finite fields}

\author[A.\ Silverberg]{A.\ Silverberg}
\address{Department of Mathematics, University of California, Irvine, CA 92697-3875, USA}
\email{asilverb@math.uci.edu}
\thanks{Silverberg was partially supported by the National 
Science Foundation.}

\author[Yu.\ G.\ Zarhin]{Yu.\ G.\ Zarhin}
\address{Department of Mathematics, Pennsylvania State University, 
University Park, PA 16802, USA}
\email{zarhin@math.psu.edu}
\thanks{Zarhin was partially supported by the Simons Foundation (grant \#246625 to Yuri Zarkhin).
}

\dedicatory{This paper is dedicated to the memory of Scott Vanstone.}

\begin{document}

\maketitle

\section{Introduction}
\label{intro}

In this paper we give conditions
under which two abelian varieties that are defined over a finite
field $F$, and are isogenous over some larger field, are 
$F$-isogenous. Further, 
we give conditions under which a given isogeny is defined over $F$.  

If $\X$ and $\Y$ are abelian varieties defined over a field $F$, and $n$
is an integer not divisible by the characteristic of $F$ and greater
than $2$, then every homomorphism between $\X$ and $\Y$ is defined over
$F(\X_n,\Y_n)$, the smallest extension of $F$ over which the $n$-torsion
points on $\X$ and $\Y$ are defined (see Theorem 2.4 of \cite{silverberg}). 
In this paper (and earlier in
\cite{silzar}) we consider the situation where we do not have full level $n$
structure, and give conditions under which isogenies between rigidified
polarized abelian varieties, with level structure given by maximal isotropic
subgroups of $n$-torsion points, are defined over fields of definition
for the rigidified polarized abelian varieties. 
Our main results are Theorems \ref{equiv}, \ref{skol}, and \ref{mprime} and Corollary \ref{mprimecor} below.

We emphasize that our assumptions
do not require that all the points on the given maximal isotropic subgroups
be defined over $F$, but only that the subgroups be defined over $F$ and
that the restrictions of the isogeny to the subgroups be an isomorphism
defined over $F$.

The results of this paper give improvements in the case of abelian varieties
over finite fields, on results in \cite{silzar} for abelian varieties over
arbitrary fields. For example, Theorem \ref{equiv} below does not hold 
in characteristic zero, even for elliptic curves.
The results in this paper rely on the theory of abelian varieties 
over finite fields (see \cite{tate}),
and our results in \cite{silzar} and \cite{serrelem}.

Abelian varieties over finite fields are useful in cryptography,
especially in the case of elliptic curves and Jacobians of
curves of genus two.
Isogenies and torsion points arise, for example,
when considering the so-called ``distortion maps'' in pairing-based cryptography, 
or when reducing the discrete log problem
on one abelian variety to the discrete log problem on an isogenous one 
(see for example \cite{Verheul}, \cite{galbraithetal},
Chapter 24 of \cite{CohenFrey}, Chapter 5 of \cite{HPS}).
These are settings where it can be important to know whether an isogeny is 
defined over the ground field,
or is defined over an extension over which a given torsion point or subgroup
is defined.

Next, we collect together hypotheses that are common to all our main results,
and refer to them collectively
as $\Phi (n)$.
(See \S\ref{notation} for notation and definitions.)

\begin{defn}
\label{Phindef}
The condition $\Phi (n)$ will mean 
that the following situation (a - g) holds:
\begin{itemize}
\item [\rm{(a)}] $\X$ and $\Y$ are positive dimensional abelian varieties 
defined over a finite field $F$,
\item [\rm{(b)}] $n$ is a positive integer not divisible 
by the characteristic of $F$,
\item [\rm{(c)}] $f\colon \X\to \Y$ is an isogeny of degree prime to $n$
that is defined over a finite extension $L$ of $F$ of degree $m$,
\item [\rm{(d)}] $\mu$ is a polarization on $\Y$, defined over $F$,
\item [\rm{(e)}] $\widetilde{\Y}_n$ is a subgroup of $\Y_n$, defined over $F$, and 
containing a maximal isotropic subgroup of $\Y_n$ with respect to the 
$e_n$-pairing induced by $\mu$,
\item [\rm{(f)}] $\widetilde{\X}_n$ is a subgroup of $\X_n$, defined over $F$, such 
that the restriction of $f$ to $\widetilde{\X}_n$ is an isomorphism from 
$\widetilde{\X}_n$ onto $\widetilde{\Y}_n$ defined over $F$,
\item [\rm{(g)}] $\lambda$ is the polarization on $\X$ defined by 
$\lambda = {}^tf\mu f$, and $\lambda$ is defined over $F$.
\end{itemize}
\end{defn}
(The set $\widetilde{\X}_n$ (resp., $\widetilde{\Y}_n$)
is a Galois submodule of $\X_n$ (resp., ${\Y}_n$),
and the restriction of $f$ to $\widetilde{\X}_n$
is an isomorphism of Galois modules
$f : \widetilde{\X}_n \isom \widetilde{\Y}_n$.)

The collection of hypotheses $\Phi (n)$ differs from the collection
$I(n)$ of our earlier work \cite{silzar} in that it contains the additional
assumptions that the field $F$ is finite and the polarizations
$\mu$ and $\lambda$ are defined over $F$. Also, the degree of $L$
over $F$ is now given the label $m$.

A special case of our results is that if $\Phi (n)$ holds for some
$n \ge 5$ with $m \le 3$, and if $\X$ (and thus $\Y$) is $L$-simple,
then $\X$ and $\Y$ are $F$-isogenous (see Corollary \ref{mprimecor}).

Additional notation is defined in \S\ref{notation} below.  
In \S\ref{prelims} we give some lemmas 
about abelian varieties over finite fields that enable us
to prove our main results in \S\ref{isogsect}.
In \S\ref{sknsect} we apply the results of \S\ref{isogsect} and the 
Skolem-Noether Theorem to obtain additional information.

We thank the referees for helpful comments.

\section{Notation and definitions}
\label{notation}

If $F$ is a field, then $\bar F$ denotes an algebraic closure of $F$, 
and $F^s \subseteq \bar{F}$ denotes a separable closure of $F$. 
A set $C$ 
(respectively, a
map $g$) that is defined over $F^s$ is defined over $F$ if and only if 
$\sigma(C) = C$ (respectively, $g^\sigma = g$) for all $\sigma\in\Gal(F^s/F)$,
where $g^\sigma(x):=\sigma(g(\sigma^{-1}(x)))$
(see p.~76 and p.~186 of \cite{weilfound}).
If $\X$ is an abelian variety defined over $F$, write $\X_n$ for 
the kernel of multiplication by $n$ in $\X({\bar F})$, write $\End(\X)$ for 
the ring of (${\bar F}$-)endomorphisms of $\X$, and let 
$\End^0(\X) = \End (\X)\otimes_{\bbz}\bbq$. 
If $n$ is not divisible by the characteristic $\fchar(F)$, then 
$\X_n$ is a free $\bbz /n\bbz$-module of rank $2 \dim (\X)$, and 
$\X_n \subseteq \X(F^s)$ (see Chapter II of \cite{MumfordAV}). 
Let $\End_F(\X)$ 
denote the ring of endomorphisms of $\X$ defined over $F$. Let 
$\End^0_F(\X) = \End_F(\X)\otimes_{\bbz}\bbq$. 

\begin{defn}
\label{ZFAdef}
We denote by $Z_F(\X)$ the center of 
the semi-simple $\bbq$-algebra $\End^0_F(\X)$.  
\end{defn}

If $\alpha$ is a homomorphism between abelian varieties that are defined 
over $F$, then $\alpha$ is defined over a finite separable extension 
of $F$ (see Corollary 1, p.~258 of \cite{chow}). If $\alpha \in \Hom (\X,\Y)$
is an isogeny, then $\alpha'$ denotes the unique element 
$\beta$ of $\Hom(\Y,\X)\otimes_{\bbz}\bbq$ such that 
$\alpha \beta = 1$ and $\beta\alpha = 1$.  
If $\X^\ast $ and $\Y^\ast $ are respectively the Picard varieties of abelian
varieties $\X$ and $\Y$, and $\alpha \in \Hom(\X,\Y)\otimes_{\bbz}\bbq$, then 
$^{t}\alpha \in \Hom(\Y^\ast ,\X^\ast )\otimes_{\bbz}\bbq$ denotes the 
transpose of $\alpha$ (see p.~124 of \cite{lang} or p.~3 of \cite{shimtan}).  
Polarizations on $\X$ will be viewed as isogenies from $\X$ onto $\X^\ast $.

If $\X$ is an abelian variety defined over a field $F$, $\lambda$ is a polarization
on $\X$, $n$ is a positive integer not divisible by $\fchar(F)$, 
and $\bmu _n$ is the $\Gal (F^s/F)$-module
of $n^{\rm th}$ roots of unity in $F^s$, then the $e_n$-pairing 
induced by 
the polarization $\lambda$ is a skew-symmetric bilinear map 
$e_{\lambda ,n} : \X_n \times \X_n \to {\bmu}_n$
(see \S 75 of \cite{weilvar} for a definition) that satisfies:
\begin{enumerate}
\item [\rm{(i)}] 
$\sigma (e_{\lambda ,n}(x_1,x_2)) = 
e_{\sigma (\lambda ),n}(\sigma (x_1),\sigma (x_2))$ if 
$\sigma \in \Gal (F^s/F)$ and $x_1$, $x_2 \in \X_n$,
\item [\rm{(ii)}] if $f : \X \to \Y$ is a homomorphism of abelian 
varieties, $\lambda$
and $\mu$ are polarizations on $\X$ and $\Y$, respectively, $\lambda =
{}^tf\mu f$, and $x_1, x_2 \in \X_n$, 
then 
$$e_{\mu ,n}(f(x_1),f(x_2)) = e_{\lambda ,n}(x_1,x_2),$$ 
\item [\rm{(iii)}] if $n$ is relatively prime to the degree of $\lambda$, 
then the pairing 
$e_{\lambda ,n}$ is non-degenerate.
\end{enumerate}

If $\X$ is an abelian variety defined over a finite field $F$, 
let $\pi_{F,\X} \in \End_F(\X)$
denote the Frobenius endomorphism of $\X$ relative to 
$F$. Then $\pi_{F,\X}$ is  invertible in $\End^0_F(\X)$.
When we are under the hypothesis $\Phi (n)$, so that an isogeny $f$ and
a finite field $L$ are given, we will let $\varphi=\varphi_f$ denote the isomorphism
\begin{equation}
\label{varphidef}
\varphi=\varphi_f\colon \End^0_L(\Y)\, \tilde{\rightarrow}\, \End^0_L(\X)
\end{equation}
defined by $\varphi(u)=f^{-1}uf$. 

\begin{defn}
An {\em $F$-isotypic abelian variety} is a positive dimensional abelian variety defined
over a field $F$ that is $F$-isogenous to a power of an $F$-simple
abelian variety.
\end{defn}

\begin{defn}
If $\X$ is an abelian variety over a field $F$, an {\em $F$-isotypic component} of
$\X$ is a maximal $F$-isotypic abelian subvariety of $\X$.  
\end{defn}

\begin{ex}
We give an example of a quadratic extension $L/F$ and an 
$F$-isotypic abelian variety $A$ such that $A$ is not $L$-isotypic\footnote{We 
have just learned about similar (but different) counterexamples 
with $F=\R$ and $\Q$ in a recent book of C-L.~Chai, B.~Conrad and F.~Oort \cite[Example 1.2.6 on pp.~22--23 and Example 1.6.4 on pp.~74--75]{CCO}.}
(contrary to Claim 10.8 of \cite[p.~146]{Oort}).
Let $L/F$ be a quadratic extension of number fields and 
let $\sigma$ be the non-trivial automorphism of $L/F$.  
Let $\O_F$ (resp., $\O_L$) denote the ring of integers in $F$ (resp., $L$).
Choose a maximal ideal $\p$ of $\O_F$ that splits into a product $\p_1 \p_2$ of two {\em distinct}
maximal ideals in $\O_L$ (the existence of such $\p$ is guaranteed by 
the Chebotarev density theorem); we have
$\sigma(\p_1)=\p_2$ and $\sigma(\p_2)=\p_1.$
Choose $\mathbf{j}\in L$  to be a $\p_1$-adic integer and not a $\p_2$-adic integer (in particular, $\mathbf{j}\not\in \O_L$)  and take an elliptic curve $E$ over $L$ whose  $j$-invariant $j(E)$ is $\mathbf{j}$.  This means  that $E$ has potentially good reduction at $\p_1$ and potentially multiplicative
reduction at $\p_2$ (see \cite[Ch. VII,  Sect. 5, Prop. 5.5]{Silverman}).
Consider the elliptic curve
$$E':=\sigma(E)=E \times_{L,\sigma}L$$
over $L$ that is the conjugate of $E$ over $F$.
Clearly, the $j$-invariant $j(E')$ of $E'$ 
is $\sigma(j(E))=\sigma(\mathbf{j})$. 
In particular,  $j(E')$ is a $\p_2$-adic integer and not a $\p_1$-adic integer.
This means that $E'$ has potentially multiplicative reduction at $\p_1$ and 
potentially good reduction at $\p_2$.
This implies that
$E$ is not isogenous to $E'$ over any extension of $L$.  
Let $A$ be the Weil restriction $R_{L/F}(E)$ of $E$ \cite[Sect.~1.3]{WeilAdeles}, which is 
an abelian surface over $F$ that is
 isomorphic over $L$
 to $E \times E'$.  In particular, $A$ is {\sl not} $L$-isotypic.
However, $A$ is a {\em simple} abelian variety over $F$.
Indeed, if $A$ were not $F$-simple, then it would contain an elliptic curve
(one-dimensional abelian subvariety)
 $E_0$ over $F$  and therefore
$$0 \ne \Hom_L(E_0,A)=\Hom_L(E_0,E\times E')=\Hom_L(E_0,E)\oplus \Hom_L(E_0,E').$$
This implies that either $\Hom_L(E_0,E) \ne 0$ or $\Hom_L(E_0,E') \ne 0$. Since $E_0$ is defined over $F$ while $E$ and $E'$ are conjugate over $F$, we conclude that  both groups $\Hom_L(E_0,E) $ and $\Hom_L(E_0,E')$ are not zero. This implies that
 $E$ and  $E'$ are isogenous over $L$, which is a contradiction. 
 It follows that $A$ is $F$-simple but not $L$-isotypic.
\end{ex}

We write $\zeta_r$ for a primitive $r$-th root of unity.

\section{Lemmas}  
\label{prelims}

The results in this section will be used in later sections to prove our
main results.

\begin{lem}
\label{newiso2}
If $F$ is a field and $A$ and $B$ are $F$-isotypic abelian varieties, then the following conditions are equivalent:
\begin{enumerate}
 \item[\rm{(i)}] $A^{\dim(B)}$ is $F$-isogenous to $B^{\dim(A)}$,
\item[\rm{(ii)}] $\Hom_F(A,B) \ne 0$,
\item[\rm{(iii)}] $\Hom_F(B,A) \ne 0$,
\item[\rm{(iv)}] $A\times B$ is $F$-isotypic.
\end{enumerate}
\end{lem}

\begin{proof}
There are $F$-simple abelian varieties $A_0$ and $B_0$ over $F$ such that $A$ is $F$-isogenous to a power of $A_0$ and $B$ is isogenous to a power of $B_0$.  Clearly, each of the above conditions is equivalent to the assertion that $A_0$ and $B_0$ are $F$-isogenous.
\end{proof}

\begin{defn}
If $A$ is a positive dimensional abelian variety over $F$, let $I_A(F)$ 
denote the set of all $F$-isotypic components of $A$. 
\end{defn}

\begin{lem}
\label{newiso}
\label{isotyprem}
Suppose $A$ is a positive dimensional abelian variety over a field $F$.
Then $Z_F(A)$ is a number field if and only if $A$ is $F$-isotypic.
\end{lem}

\begin{proof}
It is well known that if $X$ is a simple positive dimensional abelian variety 
over $F$ then $\End_F^0(X)$ is a finite-dimensional division algebra over $\Q$ and therefore its center $Z_F(X)$ is a number field. 
If $A$ is $F$-isotypic, then $A$ is $F$-isogenous to $X^n$ for some
positive dimensional $F$-simple abelian variety $X$ and some $n\in\Z_{>0}$. 
A choice of the $F$-isogeny gives rise to an algebra isomorphism between
$\End_F^0(A)$ and the matrix algebra of size $n$ over $\End_F^0(X)$. This implies that the centers of $\End_F^0(X)$ and $\End_F^0(A)$ are {\em canonically} isomorphic \cite[pp. 189--190]{ZarhinLuminy}. In particular, the center $Z_F(A)$ of $\End_F^0(A)$ is also a number field.

It follows easily from the Poincar\'e complete reducibility theorem that the set  $I_A(F)$ is finite and the natural $F$-homomorphism of abelian varieties
$$S: \prod_{X \in I_A(F)} X \to A, \ \{a_X\}_{X \in I_A(F)} \mapsto \sum_{X \in I_A(F)} a_X$$
is an $F$-isogeny.  Since $\Hom_F(X,Y)=0$ when 
$X, Y \in I_A(F)$ and $X \ne Y$, we have
$$\End_F^0(\prod_{X \in I_A(F)} X)=\bigoplus_{X \in I_A(F)} \End_F^0(X).$$
Now the isogeny $S$ induces (as in \eqref{varphidef}) an isomorphism 
$$\End_F^0(A)\cong\End_F^0(\prod_{X \in I_A(F)} X)= \bigoplus_{X \in I_A(F)} \End_F^0(X).$$
This implies that 
$$Z_F(A)\cong Z_F(\prod_{X \in I_A(F)} X)=\bigoplus_{X \in I_A(F)}Z_F(X).$$
{\em It follows that $Z_F(A)$ is a number field if and only if $\#I_A(F)=1$, 
i.e., if and only if $A$ is $F$-isotypic.}
\end{proof}

Suppose that $\X$ is an abelian variety
of positive dimension over a finite field $F$.  By
Theorem~2a on p.~140 of \cite{tate}, we have 
$Z_F(\X) = \bbq[\pi_{F,\X}]$.
Suppose 
that $L$ is a field extension of $F$ of finite degree $m$.  Since 
$\pi^m_{F,\X} = \pi_{L,\X}$, 
we have 
\begin{equation}
\label{ZFLX}
Z_F(\X) = \bbq[\pi_{F,\X}]  \supseteq   \bbq[\pi_{L,\X}] = Z_L(\X).
\end{equation}

\begin{prop}
 \label{isoFL}
Suppose $F$ and $L$ are finite fields, and $L$ is an extension of $F$.
If $A$ is an $F$-isotypic abelian variety, 
then $A$ is $L$-isotypic.\footnote{After this paper appeared on the arXiv and was submitted, we learned that this proposition was proved in 
\cite[Prop. 1.2.6.1 on p.~23]{CCO}. The proof there is completely different from ours.}
\end{prop}

\begin{proof}
Since $A$ is $F$-isotypic, $Z_F(\X)$ is a number field by Lemma \ref{newiso}.
Since $Z_F(\X)$ contains the $\Q$-(sub)algebra  $Z_L(\X)$, the latter is also a number field, which implies that $A$ is $L$-isotypic,
again by Lemma \ref{isotyprem}.
\end{proof}

\begin{defn}
Suppose $A$ is an abelian variety over a field $F$,
and $L$ is a finite extension of $F$.
We say that $X,Y \in I_A(F)$ are {\em $L$-connected} if $\Hom_L(X,Y) \ne 0$.  
\end{defn}

It follows from Lemma \ref{newiso2} that $L$-connectedness is an {equivalence relation}
 and $I_A(F)$ splits into a {\em disjoint} union of its maximal $L$-connected components. 
  
 \begin{defn}
Suppose $A$ is an abelian variety over a field $F$,
and $L$ is a finite extension of $F$.
Let $\mathbf{\pi}_0(I_A(F))$ denote the set of maximal $L$-connected components of 
$I_A(F)$.
\end{defn}

\begin{lem}
\label{oldRmk3.2ii}
Suppose $A$ is an abelian variety over a finite field $F$,
and $L$ is a finite extension of $F$.
Then:
\begin{enumerate}
\item [\rm{(i)}]
every $L$-isotypic component of $A$ is defined over $F$,
\item[\rm{(ii)}]
if $\#I_A(F) = \#I_A(L)$
then $I_A(F) = I_A(L)$.

\end{enumerate}
\end{lem}

\begin{proof}
By Proposition \ref{isoFL}, each $X \in I_A(F)$ is $L$-isotypic.  
The set 
$\mathbf{\pi}_0(I_A(F))$  
may be canonically identified with the   set $I_A(L)$ of $L$-isotypic components of $A$, namely, if $C \in \mathbf{\pi}_0(I_A(F))$ 
then $\prod_{X \in C}X$ is an $L$-isotypic abelian variety that is $F$-isogenous to the abelian subvariety
$A_{C}:=\sum_{X \in C}X$
of $A$. This implies that $A_C$ is also $L$-isotypic.
Since all $X\in C$ are defined over $F$, the abelian subvariety $A_C$ is also defined over $F$. It follows from Lemma \ref{newiso2} that $A_C$ is $L$-isotypic. 
If $C^{\prime}\in\mathbf{\pi}_0(I_A(F))$ and $C\neq C^{\prime}$, then $\Hom_L(A_C, A_{C^{\prime}})=0$.  Since
$I_A(F)$ is the disjoint union of all $C \in \mathbf{\pi}_0(I_A(F))$, we have
$A=\sum_{C \in \mathbf{\pi}_0(I_A(F))}A_C.$
This implies that each $A_C$ is an $L$-isotypic component of $A$,
and every $L$-component of $A$ is of the form $A_C$. In particular, 
every $L$-isotypic component of $A$ is  defined over $F$. 

It also follows that if $\#I_A(F) = \#I_A(L)$,
then each maximal $L$-connected component of $I_A(F)$ consists of a single element, i.e., 
each $X \in I_A(F)$ is $L$-isotypic. In other words, 
if $\#I_A(F) = \#I_A(L)$, then
every $F$-isotypic component of $A$ is also $L$-isotypic and vice versa.
\end{proof}

\begin{prop}[See Corollary 3.4 of \cite{serrelem}]
\label{serrelemma}
Suppose $n$ is an integer greater than $4$, $\O$ is an integral 
domain of characteristic zero such that no rational prime divisor of $n$ 
is a unit in $\O$, $\alpha \in \O$, $\alpha$ has finite multiplicative order, 
and $(\alpha - 1)^2 \in n\O$.
Then $\alpha = 1$.
\end{prop}

The following proposition from \cite{silzar} 
can be viewed as a variation on the Theorem on p.~17-19 of
\cite{serre}, which says that an automorphism of an abelian variety that has finite
order, and is congruent to $1$ modulo some integer $n \ge 3$,
is the identity automorphism. Proposition \ref{mink} can be proved using
Proposition \ref{serrelemma}.

\begin{prop}[Theorem 4.1b of \cite{silzar}]
\label{mink}
Suppose $\Y$ is an abelian variety, $n$ and $r$ are
relatively prime positive integers, and $n$ is at least $5$ and is not divisible
by the characteristic of a field of definition for $\Y$.  Suppose $\alpha$ is an
element of $\End(\Y)\otimes_{\bbz}\bbz[1/r]$, $\alpha$ has finite multiplicative
order, 
$\widetilde{\Y}_n$ is a subgroup of $\Y_n$ 
on which $\alpha$ induces the identity map, and $(\alpha - 1)\Y_n \subseteq
\widetilde{\Y}_n$.  Then $\alpha =1$.
\end{prop}

We now give some elementary lemmas concerning abelian varieties 
over finite fields, which we will make use of in later sections.

\begin{lem}
\label{defined}
If $\X$ and $\Y$ are abelian varieties over a finite
field $F$, and $f\colon \X\to \Y$ is a homomorphism, then $f$ is defined 
over $F$ if and only if 
$f\pi_{F,\X} = \pi_{F,\Y}f$.
\end{lem}

\begin{proof}
Let $\tau \in \Gal(F^s/F)$ denote the Frobenius element. 
By the definitions of $\tau$, $\pi_{F,\X}$, and $\pi_{F,\Y}$, we have 
$\tau(f)\pi_{F,\X} = \pi_{F,\Y}f$. 
Now $f$ is defined over $F$ if and only $\tau (f)
= f$.   Therefore,  $f$ is defined over $F$ if and only if 
$f\pi_{F,\X} = \pi_{F,\Y}f$.  
\end{proof}

\begin{lem}
\label{equal}
Suppose we have $\Phi(n)$ for some $n > 0$.  Then:
\begin{itemize}
\item [\rm{(a)}]  the elements $\varphi(\pi_{F,\Y})$ and $\pi_{F,\X}$ of
$\End^0_L(\X)$ are equal on  $\widetilde{\X}_n$, and
\item [\rm{(b)}]  the elements $\varphi^{-1}(\pi_{F,\X})$ and $\pi_{F,\Y}$ 
of $\End^0_L(\Y)$ are equal on  $\widetilde{\Y}_n$.
\end{itemize}
\end{lem}

\begin{proof}
By hypothesis (f) of assumption $\Phi(n)$, we know that
the restriction of the isogeny $f$ to the subgroup $\widetilde{\X}_n$ is defined
over $F$.  As in Lemma \ref{defined}, this means that 
$f\pi_{F,\X} = \pi_{F,\Y}f$ on 
$\widetilde{\X}_n$.  Therefore, we have (a).  Since the restriction of $f$ to
$\widetilde{\X}_n$  is an isomorphism from $\widetilde{\X}_n$ onto
$\widetilde{\Y}_n$, we have (b).  
\end{proof}

\begin{lem}
\label{centralizer}
If $\X$ is an abelian variety over a finite field
$F$, and $L$ is a finite field extension of $F$, then the centralizer of 
$Z_F(\X)$ in $\End^0_L(\X)$ is $\End^0_F(\X)$. 
\end{lem} 

\begin{proof}
The centralizer of $Z_F(\X)$ in $\End^0_L(\X)$ contains
$\End^0_F(\X)$, since $Z_F(\X)$ is the center of $\End^0_F(\X)$.  
Let $\sigma \in \Gal(L/F)$
denote the Frobenius element.  If $\beta$ is in the centralizer of $Z_F(\X)$ 
in $\End^0_L(\X)$, then 
$\beta \pi_{F,\X} = \pi_{F,\X}\beta$.
But from the definitions
of $\sigma$ and $\pi_{F,\X}$ we have 
$\sigma(\beta)\pi_{F,\X} = \pi_{F,\X}\beta$. 
Therefore, $ \sigma(\beta) = \beta$, and so $\beta \in \End^0_F(\X)$.  
\end{proof}

\section{Fields of definition for isogenies}
\label{isogsect}

In \S\S\ref{isogsect}--\ref{sknsect} we determine fields
of definition for isogenies of abelian varieties, under
the hypothesis $\Phi(n)$ (see Definition \ref{Phindef}) and certain additional conditions. 
In Theorem \ref{equiv}  
and Proposition \ref{comdef} we give
conditions under which the given isogeny $f$ is defined over $F$. In
Proposition 4.1 we deal with an intermediate field of definition, between 
the fields $F$ and $L$, for the isogeny $f$
(and the restrictions of $f$ to $L$-isotypic components of $\X$), and use this
result to prove the others.

\begin{prop}
\label{commutes}
Suppose we have $\Phi(n)$ for some $n \ge 5$, $j$ is a divisor of $m$, 
$F'$ is the degree $j$ extension of $F$ in $L$, and
$\varphi$ is defined as in \eqref{varphidef}. 
Suppose either
\begin{itemize}
\item [\rm{(i)}]   $\X' = \X$, $\Y' = \Y$, $f' = f$, and 
$\varphi' = \varphi$, or
\item [\rm{(ii)}]   $\X'$ is an $L$-isotypic component of $\X$,
$\Y' = f(\X')$, $f' : \X' \to \Y'$ is the isogeny induced by $f$, and
$\varphi' : \End^0_L(\Y') \, \tilde{\rightarrow}\, \End^0_L(\X')$ is
defined by $\varphi'(u)=(f')^{-1}uf'$.
\end{itemize}
If $(\varphi')^{-1}(\pi^j_{F,\X'})$ commutes with $\pi^j_{F,\Y'}$,
then
$\varphi'(\pi^j_{F,\Y'}) = \pi^j_{F,\X'}$ and
$f'$ is defined over $F'$.
\end{prop}

\begin{proof}
We have 
$\pi^j_{F,Z} = \pi_{F^{\prime},Z}$ 
for $Z = \X$, $\X'$, $\Y$, and $\Y'$.
Let $\sigma \in \Gal(L/F^{\prime})$ denote the Frobenius element, and
let 
$$\alpha = \varphi^{-1}(\pi_{F^{\prime},\X})\pi^{-1}_{F^{\prime},\Y} =
f\sigma(f)^{-1} \in \End^0_L(\Y).$$ 
Let 
$$\alpha' = 
(\varphi')^{-1}(\pi_{F^{\prime},\X'})\pi^{-1}_{F^{\prime},\Y'} =
f'\sigma(f')^{-1} \in \End^0_L(\Y').$$
Since $(\varphi')^{-1}(\pi_{F^{\prime},\X'})$ commutes with 
$\pi_{F^{\prime},\Y'}$, and $f'$ is defined over $L$, we have
$$(\alpha')^{m/j} = (\varphi')^{-1}(\pi_{L,\X'})\pi^{-1}_{L,\Y'} = 1.$$
Since the isogeny $f$ has degree relatively prime to $n$,
so does the isogeny $\sigma(f)$.  Therefore, 
$\alpha \in \End(\Y)\otimes_{\bbz}\bbz[1/r]$ 
for some positive integer $r$ relatively prime
to $n$.  By Lemma~2.2c of \cite{silzar}, we have 
$(\alpha - 1)\Y_n \subseteq \widetilde{\Y}_n$.  
By Lemma \ref{equal}b above, $\alpha  = 1$ on $\widetilde{\Y}_n$.  
Therefore, 
$(\alpha - 1)^2 \Y_n = 0$, 
and so 
$(\alpha' - 1)^2 (\Y')_n = 0$.
In case (i), Proposition \ref{mink} implies $\alpha  = 1$, as desired.  
Suppose $\ell$ is a prime divisor of $n$ and let 
$$T_\ell : \End^0(\Y') \to \End(V_\ell(\Y')) \cong M_{2g}(\Q_\ell)$$ 
be the natural map, where $g = \dim (\Y')$. Then $T_\ell(\alpha')$ is a
matrix of finite order, and 
$(T_\ell(\alpha') - 1)^2 \in nM_{2g}(\Z_\ell)$.
By Lemma 5.3 of \cite{semistab}, if $\lambda$ is an eigenvalue of $T_\ell(\alpha')$
then $(\lambda - 1)^2 \in n{\bar \Z}$, 
where ${\bar \Z}$ is the ring of
algebraic integers in an algebraic closure of $\Q$. By 
Proposition \ref{serrelemma}, $\lambda = 1$.
Therefore, $\alpha' = 1$. It follows that
$\varphi'(\pi_{F^{\prime},\Y'}) = \pi_{F^{\prime},\X'}$, 
and Lemma \ref{defined} 
implies $f'$ is defined over $F^{\prime}$.   
\end{proof}

\begin{prop}
\label{comdef}
Suppose we have $\Phi(n)$ for some $n \ge 5$.  Suppose hypothesis 
(i) or (ii) of Proposition \ref{commutes} holds.
If $(\varphi')^{-1}(\pi_{F,\X'})$ commutes with $\pi_{F,\Y'}$, then
$\varphi'(\pi_{F,\Y'}) = \pi_{F,\X'}$  and
$f'$ is defined over $F$.
\end{prop}

\begin{proof}
This follows directly from Proposition~\ref{commutes} with $j = 1$. 
\end{proof}

\begin{thm}
\label{equiv}
Suppose we have $\Phi(n)$ for some $n \ge 5$. 
Suppose hypothesis (i) or (ii) of Proposition \ref{commutes} holds.
Then the following are equivalent:
\begin{itemize}
\item [\rm{(a)}]  $f'$ is defined over $F$,
\item [\rm{(b)}]   $\varphi'(\pi_{F,\Y'}) = \pi_{F,\X'}$,
\item [\rm{(c)}]   $\varphi'(\End^0_F(\Y')) = \End^0_F(\X')$,
\item [\rm{(d)}]   $\varphi'(Z_F(\Y')) = Z_F(\X')$,
\item [\rm{(e)}]   $\varphi'(Z_F(\Y')) \subseteq Z_F(\X')$,
\item [\rm{(f)}]   $\varphi'(Z_F(\Y')) \supseteq Z_F(\X')$.
\end{itemize}
\end{thm}

\begin{proof}
Lemma \ref{defined} gives the equivalence of (a) and (b).  The
implications
(b) $\Rightarrow$ (d) $\Rightarrow$ (e) and (c)
$\Rightarrow$ (d) $\Rightarrow$ (f) are straightforward.  

To show that (d) $\Rightarrow$ (c), take $\gamma \in \End^0_F(\Y')$ and $\beta
\in Z_F(\X')$.  Then (d) implies that $(\varphi')^{-1}(\beta) \in Z_F(\Y')$, and
therefore 
$(\varphi')^{-1}(\beta)\gamma = \gamma(\varphi')^{-1}(\beta)$.  
Thus,
$\beta\varphi'(\gamma) = \varphi'(\gamma)\beta$, 
and so $\varphi'(\gamma)$ is in the
centralizer in $\End^0_L(\X')$ of $Z_F(\X')$.  By 
Lemma~\ref{centralizer}, 
$\varphi'(\gamma) \in \End^0_F(\X')$.  
Similarly, the reverse inclusion, 
$$\End^0_F(\X') \subseteq \varphi'(\End^0_F(\Y')),$$ 
follows from the fact that the
centralizer in $\End^0_L(\Y')$ of $Z_F(\Y')$ is $\End^0_F(\Y')$.

To show that (f) $\Rightarrow$ (b), note that (f) implies that 
$$\pi_{F,\X'} \in Z_F(\X') \subseteq \varphi'(Z_F(\Y')),$$  
and therefore 
$(\varphi')^{-1}(\pi_{F,\X'}) \in Z_F(\Y')$, so
$(\varphi')^{-1}(\pi_{F,\X'})$ commutes with $\pi_{F,\Y'}$.  We thus
have (f) $\Rightarrow$ (b) from Proposition~\ref{comdef}.

Similarly, to show that (e) $\Rightarrow$ (b), note that (e) implies that
$$\varphi'(\pi_{F,\Y'}) \in \varphi'(Z_F(\Y')) \subseteq Z_F(\X').$$  Therefore, (e)
implies that $\varphi'(\pi_{F,\Y'})$ commutes with $\pi_{F,\X'}$, and so
$\pi_{F,\Y'}$ commutes with $(\varphi')^{-1}(\pi_{F,\X'})$.  We thus
have (e) $\Rightarrow$ (b) from Proposition~\ref{comdef}.  
\end{proof}

\begin{rem}
Under assumption (i) in Propositions \ref{commutes} and \ref{comdef}
and in Theorem \ref{equiv}, we can allow $n$ to equal $4$ if we assume
in addition that $\widetilde{\Y}_4 \cong (\Z/4\Z)^b$ for some $b$.
This is because Proposition \ref{mink} holds for $n = 4$ under the
assumption $\widetilde{\Y}_4 \cong (\Z/4\Z)^b$ (see Theorem 4.1c of
\cite{silzar}).
\end{rem}

\begin{ex}
We give an example to show that one cannot drop the hypothesis that
the ground field $F$ is finite.
Let $\X = \Y = E$ be an elliptic curve over $F = \R$ with complex multiplication,
and take $g\in\End_\C(E) - \End_\R(E)$, where $\C$ and $\R$ denote the fields of
complex and real numbers, respectively.
For each $n\in \Z_{>0}$ let 
$f= f'=1+ng \in \End_\C(E)$.
Then $\End^{0}_\R(E)=\Q$, 
$E(\R)$ contains a cyclic order $n$ subgroup 
$\tilde{E}_n = \tilde{\X}_n = \tilde{\Y}_n$, 
$f$ acts on $\tilde{E}_n$ as the identity map,
and $\varphi = \varphi'$ is the identity map. 
Here, conclusions (c--f) of Theorem \ref{equiv} hold while conclusion (a) does not.
\end{ex}

\begin{ex}
Next we give an example where Theorem \ref{equiv} applies,
and statements (a--f) all fail to hold (i.e., $f'$ is not defined over $F$).
Suppose $p$ is a prime, and either $p=5$ or $p\ge 11$.
Let $\X = \Y = E$ be a supersingular elliptic curve over $F = \F_p$
(see Chapter V of \cite{Silverman} for basic facts about  
elliptic curves over finite fields, especially supersingular ones). 
Then $\#E(\F_p)=p+1$. 
Writing $ E(\F_p) \cong \Z/a\Z \times \Z/ab\Z$ with
$a,b\in\Z_{>0}$, then $E[a] \subseteq E(\F_p) = \ker(\Frob - 1)$,
where $\Frob$ denotes the Frobenius endomorphism.
By Corollary III.4.11 of \cite{Silverman} there exists
$\lambda \in \End(E)$ such that $a\lambda = \Frob -1$.
Since the characteristic polynomial of $\Frob$ is $x^2+p$, it follows that
the characteristic polynomial of $\lambda$ is $x^2+{\frac{2}{a}}x + {\frac{1+p}{a^2}} \in \Z[x]$.
Thus $a=1$ or $2$, and if $p\equiv 1\pmod{4}$ then $a=1$.
Therefore either $E(\F_p)$ is cyclic (of order $p+1$), or
$p\equiv 3\pmod{4}$ and $E(\F_p)$ has a cyclic subgroup of order $(p+1)/2$.
Let $\tilde{E}_n$ be a cyclic subgroup of $E(\F_p)$ of order $n \ge 5$ (such
exist by the assumptions on $p$; further, $p \nmid n$).
Pick $g\in\End(E) - \End_{\F_p}(E)$ and let 
$$f= f'=1+ng \in \End(E) - \End_{\F_p}(E).$$
Then $f$ acts on $\tilde{E}_n$ as the identity map,
$L=\F_{p^2}$, and $m=2$.
By Theorem \ref{equiv}, 
$\varphi(Z_{\F_p}(E))$ and $Z_{\F_p}(E)$ must be different quadratic subfields
of the quaternion $\Q$-algebra $\End^0(E)$.
For example, restricting to the case where $E$ is $y^2=x^3+x$
and $p\equiv 3\pmod{4}$, then we can write 
$$Z_{\F_p}(E) = \Q + \Q j \cong \Q(\sqrt{-p})$$ and
$$\End^0(E) = \Q + \Q i + \Q j + \Q ij$$
with $i^2 = -1$, $j^2=-p$, and $ji=-ij$.
Taking $g=i$, so that $f=1+ni$, then 
$$f^{-1} = \frac{1-ni}{n^2+1}$$
and 
$$\varphi(Z_{\F_p}(E)) = \Q + \Q \varphi(j) \cong \Q(\sqrt{-p}).$$ 
Since 
$$\varphi(j) = f^{-1}jf = \frac{(1-n^2)j-2nij}{n^2+1},$$ it follows that
$Z_{\F_p}(E) \neq \varphi(Z_{\F_p}(E))$
inside $\End^0(E)$ whenever $n>0$.
\end{ex}

\begin{rem}
\label{commutative}
It follows  directly from Proposition \ref{comdef} that if we have 
$\Phi(n)$ for some $n \ge 5$, and $\End_L(\X)$ is commutative, then $f$ is 
defined over $F$. However, this result also follows directly from
Theorem 1.3 and Remark 1.4 (see also Remark 1.5) of \cite{silzar}, which 
imply that if we have $\Phi(n)$ for some $n \ge 5$ and $\End^0_L(\X) =
\End^0_F(\X)$ then $f$ is defined over $F$ (if $\End_L(\X)$ is commutative
and $L$ is finite,
then using \eqref{ZFLX} we have 
$\End^0_L(\X) = Z_L(\X) \subseteq Z_F(\X) \subseteq \End^0_F(\X)$). 
\end{rem}

\section{Applications of the Skolem-Noether Theorem}
\label{sknsect}

In Theorem 1.6 of \cite{silzar}
we showed that if we have $\Phi(n)$ for some $n$ that does not divide
$m^2$, then $\X$ and $\Y$ are $F$-isogenous. 
In Theorems \ref{skol} and \ref{mprime} below we give other conditions on $m$
under which (with certain additional conditions on the abelian varieties) 
$\X$ and $\Y$ are $F$-isogenous. 
First we state the Skolem-Noether~Theorem, which we will use in
Proposition \ref{deltaprop}. For a proof, see Theorem~3.62 on p.~69 of 
\cite{cr}.

\begin{thm}[Skolem-Noether~Theorem]
Suppose $K$ is a field and suppose $\N$
and $\N^{\prime}$ are
simple $K$-subalgebras of a central simple $K$-algebra $\M$.  Then every
isomorphism of $K$-algebras $g \colon�\N \to \N^{\prime}$ can be
extended to an inner automorphism of $\M$, that is, there exists 
an invertible element $a \in \M$ such that $g(b) = a^{-1}ba$ for every $b \in \N$.
\end{thm}

The remaining theorems will all make use of the following
result.

\begin{prop}
\label{deltaprop}
Suppose $\X$ and $\Y$ are abelian varieties defined over a finite field $F$,
$f : \X \to \Y$ is an isogeny defined over a field extension $L$ of $F$, and
for each of $\X$ and $\Y$ the number of $F$-isotypic components
is equal to the number of $L$-isotypic components.  If $\X'$ is an
$L$-isotypic component of $\X$, let $\Y' = f(\X')$, let $f' : \X' \to \Y'$ be 
the isogeny induced by $f$, and define an isomorphism
$\varphi'$ from $\End^0_L(\Y')$ onto $\End^0_L(\X')$
by $\varphi'(u)=(f')^{-1}uf'$.
Suppose that for every $L$-isotypic component $\X'$ of $\X$, there is an
isomorphism 
$\delta \colon Z_F(\Y') \, \tilde{\rightarrow} \, Z_F(\X')$
such that $\delta(\pi_{F,\Y'}) = \pi_{F,\X'}$ 
and such that $\delta = \varphi'$ on $Z_L(\Y')$. 
Then $\X$ and $\Y$ are $F$-isogenous.
\end{prop}

\begin{proof}
The hypothesis about isotypic components implies that every $L$-isotypic
component is also an $F$-isotypic component by
Lemma \ref{oldRmk3.2ii}. Let $\X'$ be an $L$-isotypic
component of $\X$. Then $Z_L(\X')$, $Z_F(\X')$, $Z_L(\Y')$, and $Z_F(\Y')$ are
number fields by Lemma \ref{newiso}. 
The map $\delta(\varphi')^{-1}$ defines an isomorphism from
$\varphi'(Z_F(\Y'))$ onto $Z_F(\X')$.  In the Skolem-Noether Theorem, take 
$$K = Z_L(\X'), \quad \M = \End^0_L(\X'), \quad \N = \varphi'(Z_F(\Y')),
\quad {\text{ and }} \quad  \N^{\prime} = Z_F(\X').$$ 
The algebra $\M$ is a central simple $K$-algebra, since by our assumption 
$\X'$ is $F$-isogenous to a power of an $F$-simple abelian variety.  
Further, $\N$ is a simple $K$-subalgebra of $\End^0_L(\X')$, 
since $\varphi'(Z_F(\Y'))$ is isomorphic to the field $Z_F(\Y')$, which,
because $L$ and $F$ are finite, 
by \eqref{ZFLX} 
contains 
the subfield $Z_L(\Y') \cong Z_L(\X') = K$. 
The Skolem-Noether Theorem shows the existence
(after multiplying, if necessary, by a sufficiently divisible integer) of 
an isogeny
$u \in \End_L(\X')$ such that 
$\delta(\varphi')^{-1}(b) = u^{-1}bu$ 
for every $b \in \varphi'(Z_F(\Y'))$.  Therefore, 
$$\delta(a) = u^{-1}\varphi'(a)u = u^{-1}(f')^{-1}af'u$$ 
for every $a \in Z_F(\Y')$. 
Thus, $f'u\delta(a) = af'u$  for every $a \in Z_F(\Y')$,
and in particular, 
$$f'u\pi_{F,\X'} = f'u\delta(\pi_{F,\Y'}) = \pi_{F,\Y'}f'u.$$  
By Lemma \ref{defined}, the isogeny $f'u$ is defined over the field $F$. 
Summing the isogenies $f'u$ corresponding to the different $L$-isotypic 
components of $\X$, we obtain an $F$-isogeny from the 
direct sum
of the 
$L$-isotypic 
components of $\X$ to the 
direct sum
 of the $L$-isotypic components of $\Y$. Since,
under our hypotheses, $\X$ and $\Y$ are each $F$-isogenous to the  
direct sums
of their $L$-isotypic components, we obtain the desired result. 
\end{proof}

If $\X$ 
is an $F$-isotypic abelian variety over a finite field, 
and $L$ is a finite extension of $F$,
then $\X$ is $L$-isotypic by Proposition \ref{isoFL}. 
So this is clearly a case where
the number of $F$-isotypic components of $\X$
is equal to the number of $L$-isotypic components of $\X$.

\begin{thm}
\label{skol}
Suppose we have $\Phi(n)$ for some $n \ge 5$, there is a $\Q$-algebra homomorphism
from $\bbq(\zeta_m)$ to $Z_L(\Y)$ that takes $1$ to $1$, and
for each of $\X$ and $\Y$ the number of $F$-isotypic components
is equal to the number of $L$-isotypic components.  
Then $\X$ and $\Y$ are $F$-isogenous.
\end{thm}

\begin{proof}
Let $\X'$ be an $L$-isotypic
component of $\X$, let $\Y' = f(\X')$, let $f' : \X' \to \Y'$ be the isogeny
induced by $f$, and let 
$\varphi' : \End^0_L(\Y') \, \tilde{\rightarrow}\, \End^0_L(\X')$ be
defined by $\varphi'(u)=(f')^{-1}uf'$. 
Since $L$ and $F$ are finite fields, by \eqref{ZFLX} 
we have inclusions of number
fields
$$Z_L(\X') \subseteq Z_F(\X'), \quad \Q(\zeta_m) 
\subseteq Z_L(\Y') \subseteq Z_F(\Y').$$

Let $e$ be the largest divisor $j$ of $m$ such that $\pi_{L,\Y'}$ is a $j^{\rm th}$ 
power in the field $Z_L(\Y')$.  Let 
$\Pi_{\Y'} = \pi^{m/e}_{F,\Y'}$.  
Then 
$\Pi_{\Y'}^e = \pi^m_{F,\Y'} = \pi_{L,\Y'}$.  
Since some $e^{\rm th}$ root of $\pi_{L,\Y'}$ is in
$Z_L(\Y')$ by the definition of $e$, and 
$\bbq(\zeta_m) \subseteq Z_L(\Y')$, 
we know that 
$\Pi_{\Y'} \in Z_L(\Y')$.
 
We will now show that the minimal polynomial for $\pi_{F,\Y'}$ 
over $Z_L(\Y')$ is
$h(t) = t^{m/e} - \Pi_{\Y'}$.  It suffices to show that $h(t)$ is irreducible over
$Z_L(\Y')$.  Each root of $h(t)$ is the product of $\pi_{F,\Y'}$ by an $m^{\rm th}$
root of unity.  Suppose $g(t)$ is an irreducible monic factor of $h(t)$ over 
$Z_L(\Y')$. The constant term of $g(t)$ is an element of $Z_L(\Y')$ that 
is a product of roots of $h(t)$, 
so it is the product of $\pi^c_{F,\Y'}$ by an $m^{\rm th}$ root of unity,
where $c$ is the degree of the polynomial $g$.  Since $\bbq(\zeta_m) \subseteq
Z_L(\Y')$, we know that $\pi^c_{F,\Y'} \in Z_L(\Y')$.  Let $d$ be the greatest common
divisor of $c$ and $m$.  Then $\pi^d_{F,\Y'} \in Z_L(\Y')$.  But  
$(\pi^d_{F,\Y'})^{m/d} =
\pi_{L,\Y'}$, so by the definition of $e$, we now have $m/d \leq e$. 
Therefore, 
$m/e \leq d \leq c = \deg(g)$.  
Since $m/e$ is the degree of the
polynomial $h$, we have $h = g$ and $h$ is irreducible.

Let 
$\Pi_{\X'} = \pi_{F,\X'}^{m/e}$.  
The isomorphism $\varphi'$ induces an isomorphism of number
fields 
$\varphi' \colon  Z_L(\Y') \, \tilde{\rightarrow} \, Z_L(\X')$ 
such that 
$\varphi'(\pi_{L,\Y'}) = \pi_{L,\X'}$.  
Therefore $e$ is also the largest divisor $j$ of $m$ such that
$\pi_{L,\X'}$ is a $j^{\rm th}$ power in $Z_L(\X') \cong Z_L(\Y')$.
We have $\Pi_{\X'} \in Z_L(\X')$, and the above
reasoning shows that $t^{m/e} - \Pi_{\X'}$ is the minimal polynomial of 
$\pi_{F,\X'}$
over $Z_L(\X')$.

Since $\Pi_{\Y'} \in Z_L(\Y')$, we have 
$\varphi'(\Pi_{\Y'}) = \Pi_{\X'}$
by Proposition~\ref{commutes}. 
The composition of the natural isomorphisms
$$Z_F(\Y') = Z_L(\Y')(\pi_{F,\Y'}) \cong$$ $$Z_L(\Y')[t]/(t^{m/e} - \Pi_{\Y'}) \cong
Z_L(\X')[t]/(t^{m/e} - \Pi_{\X'})$$ $$\cong Z_L(\X')(\pi_{F,\X'})  = Z_F(\X'),$$
where the middle isomorphism is given by the map $\varphi'$, defines an
isomorphism $\delta \colon Z_F(\Y') \, \tilde{\rightarrow} \, Z_F(\X')$ 
such that 
$\delta(\pi_{F,\Y'}) = \pi_{F,\X'}$ 
and such that $\delta = \varphi'$ on $Z_L(\Y')$. We can now apply Proposition \ref{deltaprop}.
\end{proof}

\begin{thm}
\label{mprime}
Suppose we have $\Phi(n)$ for some $n \ge 5$, $m$ is prime,
and for each of $\X$ and $\Y$ the number of $F$-isotypic components
is equal to the number of $L$-isotypic components.  Suppose also that for
each $L$-isotypic component $\X'$ of $\X$, either $\Q(\zeta_m) 
\subseteq Z_L(\X')$, or $\Q(\zeta_m)$ and $Z_L(\X')$ are linearly disjoint over
$\Q$. 
Then $\X$ and $\Y$ are $F$-isogenous.
\end{thm}

\begin{proof}
We will apply Proposition \ref{deltaprop}. Suppose $\X'$ is an $L$-isotypic
component of $\X$, and let $\Y' = f(\X')$. Let 
$Z = Z_L(\X')$ 
and identify $Z_L(\Y')$ with $Z$ via the isomorphism 
$\varphi' : Z_L(\Y') \, \tilde{\rightarrow}\,Z$ defined by
$\varphi'(u)=(f')^{-1}uf'$ where $f' : \X' \to \Y'$ is the isogeny
induced by $f$. It suffices to show there exists an isomorphism
$\delta \colon Z_F(\Y') \, \tilde{\rightarrow} \, Z_F(\X')$
over $Z$ such that $\delta(\pi_{F,\Y'}) = \pi_{F,\X'}$. 
If $\Q(\zeta_m)
\subseteq Z$, then the proof of Theorem \ref{skol} produces the desired
isomorphism $\delta$.

Suppose $\Q(\zeta_m)$ and $Z$ are linearly disjoint over $\Q$. 
We then have $\Gal(Z(\zeta_m)/Z)=\Gal(\Q(\zeta_m)/\Q)$.
Let $\pi = \pi_{L,\X'} \in Z$. 
Then
$\pi_{F,\X'}^m =  \pi = \pi_{L,\Y'} = \pi_{F,\Y'}^m$.
Let $g(t)$ denote the minimal polynomial for $\pi_{F,\X'}$ over $Z$, and let
$d$ be the degree of $g$. The coefficient of $t^{d-1}$ in $g(t)$ is an
element of $Z$ of the form 
$$\pi_{F,\X'}\sum_{j \in J} \zeta_m^j$$
with $J$ a subset of $\{0,1,\ldots,m-1\}$. If the sum is zero, then
$J = \{0,1,\ldots,m-1\}$ and $g(t) = t^m - \pi$,
so $t^m - \pi$ is an irreducible polynomial over $Z$ whose roots include
$\pi_{F,\X'}$ and $\pi_{F,\Y'}$, and therefore there is an isomorphism
$\delta$ as desired.
So we can suppose the sum is non-zero. Therefore,  
$\pi_{F,\X'} \in Z(\zeta_m)$.

Consider the inflation-restriction exact sequence (Kummer theory)
$$H^1(\Gal(Z(\zeta_m)/Z),\bmu_m) \to H^1(\Gal({\bar \Q}/Z),\bmu_m) 
\to H^1(\Gal({\bar \Q}/Z(\zeta_m)),\bmu_m).$$
Since $m$ is prime, we have
$\#\Gal(Z(\zeta_m)/Z) = m - 1$.
Since $\#\bmu_m = m$ and $\gcd(m,m-1) = 1$, we have
$$H^1(\Gal(Z(\zeta_m)/Z),\bmu_m) = 0.$$
The cohomological long exact sequence now tells us that the natural group homomorphism
$$Z^\times/(Z^\times)^m \hookrightarrow Z(\zeta_m)^\times/(Z(\zeta_m)^\times)^m$$
is an inclusion.
We have $\pi \in Z^\times$. Since $\pi_{F,\X'} \in Z(\zeta_m)^\times$
and $\pi_{F,\X'}^m = \pi$, we have $\pi \in (Z(\zeta_m)^\times)^m$.
Therefore, we can write $\pi = v^m$ for some $v \in Z^\times$. Then
$\pi_{F,\X'} = v\zeta_m^a$ and $\pi_{F,\Y'} = v\zeta_m^b$ with $a, b \in
\{0,1,\ldots,m-1\}$. If $a = 0$ then
$\pi_{F,\X'} \in Z \subseteq Z_F(\Y')$,
and by Theorem \ref{equiv} we have 
$Z_F(\Y') = Z_F(\X') = Z$,
and we are done. We proceed similarly if $b = 0$. Now suppose $a$ and $b$
are both non-zero. Then
$Z(\pi_{F,\X'}) = Z(\zeta_m) = Z(\pi_{F,\Y'})$,
and (since $\Q(\zeta_m)$ and $Z$ are linearly disjoint over $\Q$) there is
an automorphism of $Z(\zeta_m)$ over $Z$ that takes $\zeta_m^a$ to $\zeta_m^b$,
giving the desired isomorphism $\delta$.
\end{proof}

As an immediate corollary we have:

\begin{cor}
\label{mprimecor}
Suppose we have $\Phi(n)$ for some $n \ge 5$, 
for each of $\X$ and $\Y$ the number of $F$-isotypic components
is equal to the number of $L$-isotypic components, and $m \le 3$.
Then $\X$ and $\Y$ are $F$-isogenous.
\end{cor}

\begin{proof}
In the notation of Theorem \ref{mprime}, 
if $m\le 2$ then $\Q(\zeta_3)=\Q\subset Z$. If $m=3$, 
then $\Q(\zeta_3)/\Q$ is quadratic, so every number field (including $Z$) either contains (a subfield isomorphic to) $\Q(\zeta_3)$ or is linearly disjoint from it over $\Q$.
In both cases, the desired result follows from Theorem \ref{mprime}.
\end{proof}

\end{document}